\documentclass[11pt,english,american]{article}
\usepackage[T1]{fontenc}
\usepackage[latin9]{inputenc}
\usepackage{geometry}
\usepackage{refstyle}
\usepackage{amsmath}
\usepackage{amsthm}
\usepackage{amssymb}
\usepackage[all]{xy}
\PassOptionsToPackage{normalem}{ulem}
\usepackage{ulem}

\makeatletter

\AtBeginDocument{\providecommand\thmref[1]{\ref{thm:#1}}}
\AtBeginDocument{\providecommand\remref[1]{\ref{rem:#1}}}
\AtBeginDocument{\providecommand\lemref[1]{\ref{lem:#1}}}
\AtBeginDocument{\providecommand\noteref[1]{\ref{note:#1}}}
\AtBeginDocument{\providecommand\corref[1]{\ref{cor:#1}}}
\AtBeginDocument{\providecommand\secref[1]{\ref{sec:#1}}}
\AtBeginDocument{\providecommand\defref[1]{\ref{def:#1}}}
\AtBeginDocument{\providecommand\conjref[1]{\ref{conj:#1}}}
\AtBeginDocument{\providecommand\propref[1]{\ref{prop:#1}}}
\RS@ifundefined{subsecref}
  {\newref{subsec}{name = \RSsectxt}}
  {}
\RS@ifundefined{thmref}
  {\def\RSthmtxt{theorem~}\newref{thm}{name = \RSthmtxt}}
  {}
\RS@ifundefined{lemref}
  {\def\RSlemtxt{lemma~}\newref{lem}{name = \RSlemtxt}}
  {}

\theoremstyle{plain}
\newtheorem{thm}{\protect\theoremname}[section]
  \theoremstyle{remark}
  \newtheorem{rem}[thm]{\protect\remarkname}
  \theoremstyle{plain}
  \newtheorem{assumption}[thm]{\protect\assumptionname}
  \theoremstyle{definition}
  \newtheorem{defn}[thm]{\protect\definitionname}
  \theoremstyle{remark}
  \newtheorem{notation}[thm]{\protect\notationname}
  \theoremstyle{remark}
  \newtheorem{note}[thm]{\protect\notename}
  \theoremstyle{plain}
  \newtheorem{lem}[thm]{\protect\lemmaname}
  \theoremstyle{plain}
  \newtheorem{prop}[thm]{\protect\propositionname}
  \theoremstyle{plain}
  \newtheorem{cor}[thm]{\protect\corollaryname}
  \theoremstyle{plain}
  \newtheorem{conjecture}[thm]{\protect\conjecturename}
  \theoremstyle{definition}
  \newtheorem{example}[thm]{\protect\examplename}

\AtBeginDocument{
  
}

\makeatother

\usepackage{babel}
  \addto\captionsamerican{\renewcommand{\assumptionname}{Assumption}}
  \addto\captionsamerican{\renewcommand{\conjecturename}{Conjecture}}
  \addto\captionsamerican{\renewcommand{\corollaryname}{Corollary}}
  \addto\captionsamerican{\renewcommand{\definitionname}{Definition}}
  \addto\captionsamerican{\renewcommand{\examplename}{Example}}
  \addto\captionsamerican{\renewcommand{\lemmaname}{Lemma}}
  \addto\captionsamerican{\renewcommand{\notationname}{Notation}}
  \addto\captionsamerican{\renewcommand{\notename}{Note}}
  \addto\captionsamerican{\renewcommand{\propositionname}{Proposition}}
  \addto\captionsamerican{\renewcommand{\remarkname}{Remark}}
  \addto\captionsamerican{\renewcommand{\theoremname}{Theorem}}
  \addto\captionsenglish{\renewcommand{\assumptionname}{Assumption}}
  \addto\captionsenglish{\renewcommand{\conjecturename}{Conjecture}}
  \addto\captionsenglish{\renewcommand{\corollaryname}{Corollary}}
  \addto\captionsenglish{\renewcommand{\definitionname}{Definition}}
  \addto\captionsenglish{\renewcommand{\examplename}{Example}}
  \addto\captionsenglish{\renewcommand{\lemmaname}{Lemma}}
  \addto\captionsenglish{\renewcommand{\notationname}{Notation}}
  \addto\captionsenglish{\renewcommand{\notename}{Note}}
  \addto\captionsenglish{\renewcommand{\propositionname}{Proposition}}
  \addto\captionsenglish{\renewcommand{\remarkname}{Remark}}
  \addto\captionsenglish{\renewcommand{\theoremname}{Theorem}}
  \providecommand{\assumptionname}{Assumption}
  \providecommand{\conjecturename}{Conjecture}
  \providecommand{\corollaryname}{Corollary}
  \providecommand{\definitionname}{Definition}
  \providecommand{\examplename}{Example}
  \providecommand{\lemmaname}{Lemma}
  \providecommand{\notationname}{Notation}
  \providecommand{\notename}{Note}
  \providecommand{\propositionname}{Proposition}
  \providecommand{\remarkname}{Remark}
\providecommand{\theoremname}{Theorem}

\begin{document}
\global\long\def\kk{\mathcal{\mathbb{K}}}
\global\long\def\ff{\mathbb{F}}
\global\long\def\ll{\mathbb{L}}
\global\long\def\fxg{\ff\left\langle X_{G}\right\rangle }
\global\long\def\fxx{\ff\left\langle X\right\rangle }
\global\long\def\endo{\mbox{End}}
\global\long\def\gg{\mathfrak{g}}
\global\long\def\jj{\mathfrak{J}}

\title{$G$-Graded Central Polynomials and $G$-Graded Posner's Theorem}

\author{Yakov Karasik}
\maketitle
\begin{abstract}
Let $\ff$ be characteristic zero field, $G$ a residually finite
group and $W$ a $G$-prime and PI $\ff$-algebra. By constructing
$G$-graded central polynomials for $W$, we prove the $G$-graded
version of Posner's theorem. More precisely, if $S$ denotes all non-zero
degree $e$ central elements of $W$, the algebra $S^{-1}W$ is $G$-graded
simple and finite dimensional over its center. 

Furthermore, we show how to use this theorem in order to recapture
the result of Aljadeff and Haile stating that two $G$-simple algebras
of finite dimension are isomorphix iff their ideals of graded identities
coincide.
\end{abstract}

\section{Introduction}

\selectlanguage{english}%
Let $W$ be an associative algebra over a field $\ff$. The most basic
setup of the theory of associative algebras satisfying a polynomial
identity (PI in short) assumes the existence of a non-zero polynomial
$f=f(x_{1},...,x_{n})$ with non-commuting variables $x_{1},...,x_{n}$,
such that $f(a_{1},...,a_{n})=0$ for every $a_{1},...,a_{n}\in W$.
In that case we say that $W$ is a PI algebra and $f$ is a polynomial
identity of $A$. The class of PI algebras includes all commutative
algebras (since clearly $[x_{1},x_{2}]$ is an identity), finite dimensional
algebras, Grassmann algebras and Grassmann envelopes of finite dimensional
algebras. This family is rigid in the sense of being closed to major
algebraic operations. For instance, the homomorphic image of a PI
algebra is easily seen to be also PI. If $A\subseteq B$ and $B$
is PI so does $A$. The direct product and even the tensor product
of two PI algebras is, by a celebrated theorem of Regev \cite{Regev1971},
also PI. 

PI algebras possesses properties similar to finite dimensional algebras.
The Jacobson radical of an affine PI algebra is nilpotent \cite{AlexeiBelov-Kanela}.
If $A$ is a PI algebra containing no nilpotent ideals, then it also
does not contain nil ideals. The primitive PI algebras are simple
and finite dimensional over their center of dimension bounded by the
degree of a multilinear identity. Furthermore, if $A$ is a \emph{semiprime}
(i.e. the intersection of its prime ideals is zero which is equivalent
to $A$ having no nilpotent ideals) PI algebra, then every non-zero
ideal of $A$ intersects the center of $A$ non-trivially. The last
statement yields the famous theorem of Posner (see all in chapter
1.11 of \cite{Giambruno2005}):
\begin{thm}[Posner's Theorem]
\label{thm:PosnerUngraded}Suppose $A$ is semiprime and PI algebra
having a field as its center. Then, $A$ is a simple finite dimensional
algebra over its center.
\end{thm}
A more general version of this theorem is
\begin{thm}
Suppose $A$ is semiprime and PI algebra with a center $C$. Then,
every non-zero ideal of $A$ intersects non-trivially $C$.
\end{thm}
\selectlanguage{american}%
The main idea of the proof involves the existence of a central polynomial
for $M_{n}(\ff)$ - $n\times n$ matrices over the field $\ff$ and
the fact, due to Amitsur, that if $A$ is semiprime, then $A[x]=A\otimes_{\ff}\ff[x]$
is semisimple.

In this paper we consider the group graded analogue of Posner's Theorem.
To be more precise, suppose $G$ is a \textbf{residually finite} group.
We prove:
\begin{thm}[$G$-graded Posner's Theorem]
\label{thm:MainTheorem}Suppose $A$ is a $G$-graded algebra over
a field $\ff$ of characteristic zero. Suppose further that $A$ is
$G$-semiprime and (ordinary) PI. Then, every $G$-graded ideal $\{0\}\neq I$
of $A$ intersects the $e$-part of the center of $A$ (denoted by
$C_{e}$) in a non-trivial fashion. 

As a result, if $C_{e}$ is a field, then $A$ is a $G$-simple algebra.
Moreover, $A$ is finite dimensional over $C_{e}$.
\end{thm}
\begin{rem}
In a (English translation of) paper of Balaba \cite{MR2072615} it
seems that the author have proved a version of \thmref{MainTheorem}
for every $G$. However, the proof given there is valid only for finite
abelian groups $G$. Indeed, the author starts his proof (Proposition
2) by using Corollary 4.5 from \cite{Cohen1984} which is valid for
finite groups (it is false for infinite groups, see \remref{OnlyFinite}).
The end of the proof makes sense only for abelian groups. Indeed,
the auther claims that given an ungraded central polynomial $q(x_{1},...,x_{n})$
of a $G$-graded algebra $A$, it is possible to find $a_{1}^{0},...,a_{n}^{0}\in A$
- a homogenous evaluation of $q$ such that $q(a_{1}^{0},...,a_{n}^{0})$
is non-zero and homogenous. If $G$ is abelian, then every monomial
$a_{\sigma(1)}^{0}\cdots a_{\sigma(n)}^{0}$ is of the same homogenous
degree, so the above statment holds. However, if $G$ is not abelian
there is no reason for it to hold.

In fact, the above difficulties are the main focus of this paper.
They require a new technic to construct $G$-graded central polynomials
of $G$-simple and finite dimensional algebras.
\end{rem}
Finally, in the final section, we show how to use the above theorem
to recover a result of Aljadeff and Haile \cite{Aljadeff2014} which
states
\begin{thm}
Suppose $G$ is a finite group and $A_{1},A_{2}$ are $G$-simple
algebras of finite dimension over an algebraically closed field $\ff$
of characteristic zero. Then, $id_{G}(A_{1})=id_{G}(A_{2})$ if and
only if $A_{1}$ and $A_{2}$ are $G$-isomorphic over $\ff$.
\end{thm}
Our technique uses very little from the structure theory of $G$-simple
algebras and will be used in consequent papers in order to investigate
the connection between $G$-simple structures and their graded identities.
 
\selectlanguage{english}%

\section{\label{sec:-graded-structure-theorems}$G$-graded structure theorems}

Throughout this section $\ff$ will denote a field of \textbf{any
characteristic} and $G$ will denote \textbf{any} group. The goal
of this section is to prove the $G$-graded analogs of the basic structure
theorems of primitive and simple PI algebras. More precisely, we will
prove the $G$-graded Wedderburn's theorem, $G$-graded Jacobson's
density theorem and Kaplansky's theorem on primitive PI algebras.
These theorems even in the $G$-graded case should be well known,
however the author has failed to find a full account of their proofs
in the literature. 

We note that the work of Bakhturin, Segal and Zaicev (see \cite{MR2488221})
treats the $G$-graded Wedderburn's theorem in the case where $G$
is finite and the characteristic of $\ff$ is zero or coprime to $|G|$. 

We follow the exposition of the ungraded case in \cite{Herstein1994}.
\begin{assumption}
Throughout the chapter $A$ denotes a $G$-graded $\ff$-algebra.
\end{assumption}
\begin{defn}
Let $M$ be a left $A$-module. We say that $M$ is \emph{$G$-graded
$A$-module} if 
\[
M=\bigoplus_{g\in G}M_{g}
\]
 ($M_{g}$ is an $\ff$-vector space) and $A_{g}M_{h}\subseteq M_{gh}$
for all $g,h\in G$.

A $G$-graded $A$-module $M$ is called \emph{$G$-graded irreducible}
if it contains no non-trivial $G$-graded $A$-submodules. 

Let $M$ and $N$ be $G$-graded $A$-modules. We say that the $A$-module
map $\phi:M\to N$ is a \emph{$G$-graded homomorphism of degree $g\in G$},
if for every $h\in G$: $\phi(M_{h})\subseteq N_{gh}$. 
\end{defn}
\begin{notation}
Since $G$ and $A$ are fixed we will use the abbreviated phrases
``graded module'', ``graded irreducible'' and ``graded homomorphisms''. 
\end{notation}
\begin{note}
The kernel of a graded module homomorphism $\phi:M\to N$ is a graded
module. Indeed, suppose $a=\sum_{g\in G}a_{g}\in M$ ($a_{g}\in M_{g})$
is mapped to zero. Since each $\phi(a_{g})$ is an element of a distinct
homogeneous component of $N$, the $a_{g}$ must be mapped to zero. 

Moreover, the image of such a homomorphism is a graded submodule of
$N$.
\end{note}
Suppose $M$ is a graded module. Let $\endo{}_{A}^{G}(M)$ denote
the $\ff$-algebra generated by all graded endomorphisms of $M$.
This is a $G$-graded algebra where $\endo{}_{A}^{G}(M){}_{g}$ consists
of all graded homomorphisms of degree $g$. This indeed defines a
$G$-grading on $\endo{}_{A}^{G}(M)$: It is trivial that $\endo{}_{A}^{G}(M)_{g}\cdot\endo{}_{A}^{G}(M)_{g^{\prime}}\subseteq\endo{}_{A}^{G}(M)_{gg^{\prime}}$. 
\begin{rem}
If $G$ is a finite group, then $\endo{}_{A}^{G}(M)=\endo{}_{A}(M)$.
Indeed, if $\phi\in\endo{}_{A}(M)$, write $\phi_{g}=\sum_{h\in G}P_{gh}\phi P_{h}$,
where $P_{h}$ is the projection onto $M_{h}$, and notice that $\phi_{h}$
sits inside $\endo{}_{A}(M)_{h}$. Now, 
\[
\sum_{g\in G}\phi_{g}=\sum_{h,g\in G}P_{gh}\phi P_{h}=\left(\sum_{h\in G}P_{h}\right)\cdot\phi\cdot\left(\sum_{g\in G}P_{g}\right)=\phi.
\]
\end{rem}
\begin{defn}
A $G$-graded $\ff$-algebra $\mathbf{D}$ is called \emph{$G$-division
algebra} if every non-zero homogeneous element of $\mathbf{D}$ is
invertible.
\end{defn}
\begin{rem}
It is clear that $1\in\mathbf{D}_{e}$, so in particular $\mathbf{D}_{e}\neq0$.
\end{rem}
\begin{lem}[Schur's Lemma]
 \label{lem:Schur}If $M$ is graded irreducible, then $\endo{}_{A}^{G}(M)$
is a $G$-division algebra. Moreover, if $N$ is an irreducible graded
module which is non-graded isomorphic to $M$, then there are no graded
homomorphisms from $M$ to $N$ but the zero homomorphism. 
\end{lem}
\begin{proof}
Let $0\neq\phi$ be a homogeneous element in $\endo{}_{A}^{G}(M)$.
Since $\ker\phi$ is a graded submodule of $M$, we have $\ker\phi=0$.
So $\phi$ is injective. From the same reason its image must be equal
to $M$. So $\phi$ is also surjective. All in all, $\phi$ is isomorphism,
thus invertible. 

For the second part, let $\phi$ be a graded homomorphism between
$M$ and $N$. As before, the kernel and the image must be zero or
everything. If the kernel is zero and the image is $N$, we get that
$\phi$ is a graded isomorphism - contradicting the assumption on
$M$ and $N$. In any other case $\phi$ must be the zero homomorphism. 
\end{proof}
\begin{lem}
\label{lem:graded_devision}\label{lem:StructureOfGDivision}If $\mathbf{D}$
is a $G$-division $\ff$-algebra, then $\mathbf{D}=\oplus_{h\in H}(\mathbf{D}_{h}=Db_{h})$,
where $D$ is an $\ff$-division algebra, $H$ is a subgroup of $G$
and $b_{h}$ is an invertible element of $\mathbf{D}_{h}$. Moreover,
if $\ff$ is algebraically closed and $\mbox{dim}_{\ff}\mathbf{D}_{e}<\infty$,
then $\mathbf{D}$ is a twisted group algebra. That is, $\mathbf{D}=\ff^{\alpha}H$,
where $\alpha\in Z^{2}(H,\ff^{\times})$ ($H$ acts trivially on $\ff^{\times}$).
\end{lem}
\begin{proof}
Only the second part requires a proof. We first remark that we will
rely heavily on the algebraic closeness of $\ff$.

$\mathbf{D}_{e}$ is a division algebra over $\ff$. Moreover, the
center of $\mathbf{D}_{e}$ must be a finite extension of $\ff$.
Hence it is equal to $\ff$. Thus $\mathbf{D}_{e}$ is a central simple
$\ff$-algebra. So $\mathbf{D}_{e}=\ff$. 

Let $H$ be the subset of $G$, such that $\mathbf{D}_{h}\neq0$ for
$h\in H$. If $0\neq b_{h}\in\mathbf{D}_{h}$, then multiplication
by $b_{h}$ from (say) the right is an $\ff$-linear homomorphism
from $\mathbf{D}_{e}$ to $\mathbf{D}_{h}$. Since $b_{h}$ is invertible,
this homomorphism is invertible. Thus $\mathbf{D}_{h}=\ff b_{h}$. 

Since the product of two invertible elements is also invertible, we
deduce that $H$ is a subgroup of $G$. 
\end{proof}
\selectlanguage{american}%
\begin{defn}
A $G$-graded algebra $A$ is called $G$\emph{-primitive} if there
is a $G$-graded module $V$ of $A$ such that the action of $A$
on $V$ is faithful and $G$-irreducible.
\end{defn}
By \lemref{graded_devision}, $\endo{}_{A}^{G}(V)=\oplus_{g\in H}Db_{g}=:\mathbf{D}$,
where $H$ is a subgroup of $G$, $D$ is a division $\ff$-algebra
and $b_{g}$ is an invertible element of $\mathbf{D}$. Notice that
$H=\mbox{supp}\mathbf{D}$ (the support of a $G$-graded algebra $A$
is the subgroup of $G$ generated by all $h\in G$ for which $A_{g}\neq\{0\}$).
\begin{thm}[$G$-graded density]
\label{thm:Density}Suppose $A$ is $G$-primitive $\ff$-algebra
acting on the $G$-graded module $V$ irreducibly and faithfully.
Let $\mathbf{D}:=\endo{}_{A}^{G}(V)$, $D=\mathbf{D}_{e}$, $H=\mbox{supp}\mathbf{D}$
and $V_{T}=\bigoplus_{t\in T}V_{t}$, where $T$ is a set of representatives
of right cosets of $H$ inside $G$. 

Then, for every homogeneous elements $v_{1},....,v_{m}\in V_{T}$
which are $D$-independent and for every homogeneous $u_{1},...,u_{m}\in V$,
such that $g_{0}:=(\deg u_{i})^{-1}\cdot\deg v_{i}$ is the same for
$i=1...m$, there is an $a\in A$ such that $av_{i}=u_{i}$ ($i=1...m$).
\end{thm}
\begin{note}
It is clear that if $a=\sum_{g\in G}a_{g}$ is the decomposition of
$a$ into its homogeneous components, then $a_{g}v_{i}=0$ for $i=1...m$
and $g\neq g_{0}$. Thus, we may replace $a$ by $a_{g_{0}}$ and
thus assume that $a$ is homogeneous. 
\end{note}
\begin{proof}
We claim that it is enough to establish the following: Suppose $U$
is a finite dimensional (over $D$) $G$-graded subspace of $V_{T}$
and $v^{\prime}\notin U$ is a homogeneous element of $V_{T}$, then
there is a homogeneous $a\in A$ such that $aU=0$ but $av^{\prime}\neq0$.

To see this we proceed by induction on $m$. Let $U=\mbox{Span}_{D}\{v_{1},...,v_{m}\}$
and $v^{\prime}=v_{m+1}$. By induction we can find $a_{1}\in A$
such that $a_{i}v_{i}=u_{i}$ for $i=1...m$. Moreover, let $a_{2}\in A$
be a homogeneous element such that $a_{2}U=0$ and $a_{2}v^{\prime}\neq0$.
Since $A$ acts $G$-irreducibly on $V$ and $a_{2}v^{\prime}$is
homogeneous, it is possible to find $a_{3}\in A$ for which $a_{3}\left(a_{2}v^{\prime}\right)=u_{m+1}-a_{1}v^{\prime}$.
Thus, $a=a_{3}a_{2}+a_{1}$ moves each $v_{i}$ to $u_{i}$. 

We prove the claim by induction on $m=\dim_{D}U$. The case $m=0$
follows from faithfulness of the action of $A$ on $V$. Suppose $m>0$.
Hence, there is a homogeneous $v\in U$ and a subspace $U_{0}$ of
$U$ such that $U=U_{0}+Dv_{0}$ and $v_{0}\notin U_{0}$. 

By induction, the annihilator of $V_{0}$, $\mbox{ann}_{A}(U_{0})$,
does not annihilate any element of $V_{T}$ not in $U_{0}$. Moreover,
$\mbox{ann}_{A}(U_{0})$ is a left $G$-graded ideal of $A$, hence
by $G$-irreducibility, for every homogeneous $v\notin U_{0}$, $\mbox{ann}_{A}(U_{0})v=V$. 

Assume to the contrary that $\mbox{ann}_{A}(U)v^{\prime}=0$.We define
a map $\tau:V\to V$ by the following procedure: For a homogeneous
$0\neq v\in V$ we can always find a \uline{homogeneous} $a\in\mbox{ann}_{A}(U_{0})$
such that $av_{0}=v$. Using $a$ we declare $\tau(v)=av^{\prime}$.
This is well defined, since if $av_{0}=0$, then $a$ must be in $\mbox{ann}_{A}(U)$.
Hence, by our assumption, $a$ annihilates $v^{\prime}$, that is
$av^{\prime}=0$. Moreover, for every homogeneous $v$, $\left(\deg\tau(v)\right)^{-1}\deg v=\left(\deg v^{\prime}\right)^{-1}\deg v_{0}=:h_{0}$.
Hence $\tau$ is a homogeneous element of $\mathbf{D}.$ 

The map $\tau$ is clearly $\ff$-linear. Moreover, if $b\in A$ is
homogeneous, $bav_{0}=bv$, hence 
\[
\tau(bv)=bav^{\prime}=b\tau(v).
\]
Thus, $\tau\in\mathbf{D}$ and therefore equals to $\alpha b_{h_{0}}$,
where $\alpha\in D$. Furthermore, for $a\in A$: 
\[
av^{\prime}=\tau(av_{0})=a\tau(v_{0})\Rightarrow a(v^{\prime}-\tau(v_{0}))=0.
\]
Thus, 
\[
v^{\prime}=\tau(v_{0})\Rightarrow v^{\prime}=\alpha b_{h_{0}}v_{0}.
\]
Since both $v_{0}$ and $v^{\prime}$ are in $V_{T}$, we deduce that
$v^{\prime}=\alpha_{e}v_{0}$, which is absurd. 
\end{proof}
We turn to consider $G$-graded $G$-simple algebras. 
\begin{defn}
$A$ is $G$\emph{-simple} if $A^{2}\neq0$ and every $G$-graded
two-sided ideal of $A$ is trivial.
\end{defn}
It is easy to check that every $G$-simple algebra is also $G$-primitive.
Indeed, Choose $V=A$. The action is, by definition, is $G$-graded
irreducible. Moreover, if there was a homogeneous $a\in A$ such that
$aA=0$, then $Aa$ was a two-sided $G$-graded ideal of $A$. Therefore,
$aA=0$ or $A$. The latter is impossible since then $A^{2}=AaA=0$.
So $Fa$ is a two-sided $G$-graded ideal of $A$. As before, the
conclusion is that $\ff a=0\Rightarrow a=0$. Thus, if $a\in A$ is
not necessary a homogeneous element it is impossible that $aA=0$.
All in all, the action is also faithful.

In order to give an example of a $G$-simple algebra, we define:
\selectlanguage{english}%
\begin{defn}
\label{def:matrix-grading}Let $B$ be any $G$-graded $\ff$-algebra,
a natural number $n$ and $\gg=(g_{1},...,g_{n})\in G^{n}$. Denote
by $M_{\gg}(B)$ the $\ff$-algebra $B\otimes M_{n}(\ff)$, $G$-graded
by 
\[
M_{\mathfrak{\gg}}(B)_{g}=\mbox{Span}_{\ff}\{b\otimes e_{i,j}|\,b\in B_{h},\,g=g_{i}^{-1}hg_{j}\}.
\]
\end{defn}
\begin{note}
\label{note:MgBSimple}It is easy to check that if $B$ is $G$-simple,
then so does $M_{\gg}(B)$.
\end{note}
\selectlanguage{american}%
\begin{defn}
Let $A$ be a $G$-graded $\ff$-algebra. Denote by $A^{op}$ \emph{the
opposite $G$-graded $\ff$-algebra of $A$} defined by $A^{op}=A$
as $\ff$-vector spaces; the multiplication is given by $a_{1}a_{2}=a_{2}\cdot a_{1}$,
where $\cdot$ is the multiplication in $A$; the $G$-grading is
given via: 
\[
A_{g}^{op}=A_{g^{-1}}
\]
for all $g\in G$.
\end{defn}
\begin{lem}
\label{lem:AOppositeSimple}Let $A$ be $G$-simple $\ff$-algebra.
Then, $A^{op}$ is also $G$-simple. Furthermore, if $A$ is a $G$-division
$\ff$-algebra, so does $A^{op}$.
\end{lem}
\begin{proof}
The proof is omitted. 
\end{proof}
\selectlanguage{english}%
\begin{notation}
For an algebra $A$, we denote by $Z(A)$ the center of $A$. Moreover,
If $A$ is $G$-graded, the center $Z(A)$ will (usually) \textbf{not}
be graded by $G$, unless $G$ is abelian. Nevertheless, we denote
by $Z(A)_{e}$ all elements of degree $e$ inside $Z(A)$. 
\end{notation}
\begin{defn}
A finite dimensional $G$-simple algebra $A$ over $\ff$ is called
\emph{$G$-central-simple algebra }if $\ff=Z(A)_{e}$.
\end{defn}
\begin{prop}
\label{prop:tensorGSimple}Suppose $A$ is a $G_{1}$-central-simple
algebra over $\ff$. If $B$ is $G_{2}$-simple unital algebra over
$\ff$, then $A\otimes_{\ff}B$ is $G_{1}\times G_{2}$-central simple
algebra over $Z(B)_{e}$ (here, $(A\otimes B)_{(g_{1},g_{2})}=A_{g_{1}}\otimes B_{g_{2}}$).
In particular, if $B$ is a field (considered graded by $\{e\}$),
then $A\otimes_{\ff}B$ is $G$-simple algebra over $Z(A\otimes_{\ff}B)_{e}=B$.
\end{prop}
\begin{proof}
If $0\neq I$ is a $G_{1}\times G_{2}$-graded ideal of $A\otimes B$,
and $0\neq x=\sum_{j=1}^{t}a_{j}\otimes b_{j}\in I_{(g_{1},g_{2})}$
is shortest element in $I$ (that is, for every $0\neq x^{\prime}=\sum_{j=1}^{t^{\prime}}a_{j}^{\prime}\otimes b_{j}^{\prime}\in I$,
one has $t^{\prime}\geq t$). 

It is easy to see that $b_{1},...,b_{t}$ are independent over $\ff$.
Furthermore, since $Aa_{1}A=A$, one can write $\sum_{p}c_{p}^{\prime}a_{1}c_{p}^{\prime\prime}=1$.
We can assume that $\deg c_{p}^{\prime}\cdot g_{1}\cdot\deg c_{p}^{\prime\prime}=e$
for all $p$. Hence, 
\[
\sum_{p}c_{p}^{\prime}\otimes1\cdot x^{\prime}\cdot c_{p}^{\prime\prime}\otimes1=1\otimes b_{1}+d_{2}\otimes b_{2}+\cdots+d_{t}\otimes b_{t},
\]
where $d_{j}=\sum_{p}c_{p}^{\prime}a_{j}c_{p}^{\prime\prime}\in A_{e}$.
Thus, for every $a\in A$, 
\[
[a,1\otimes b_{1}+d_{2}\otimes b_{2}+\cdots+d_{t}\otimes b_{t}]\in I
\]
 is of length shorter than $t$, hence $[a,d_{j}]=0$ for $j=2,...,t$.
As a result, $d_{j}\in\ff$. Hence, 
\[
z=1\otimes(b_{1}+d_{2}b_{2}+\cdots+d_{t}b_{t}).
\]
Since $b_{1},...,b_{t}$ are linearly independent over $\ff$, we
get that $z\neq0$. As a result, $t=1$ and (may assume) $x=1\otimes b_{1}$.
The $G$-simplicity of $B$ and the fact that $b_{1}$ is homogenous
forces $A\otimes B\cdot x\cdot A\otimes B$. All in all, $I=A\otimes B$.

Next, it is clear that $Z(B)_{e}\subseteq Z(A\otimes B)_{e}$. For
the other direction, 
\[
Z(A\otimes B)_{e}\ni0\neq z=\sum_{i=1}^{t}a_{i}\otimes b_{i},
\]
where $t$ is minimal. As before, $b_{1},...,b_{t}$, are independent
over $\ff$. For $a\in A$, 
\[
0=[z,a]=\sum_{i=1}^{t}[a_{i},a]\otimes b_{i}.
\]
Thus, $[a_{i},a]=0$ for $i=1,...,t$. Hence, $a_{i}\in Z(A)_{e}=\ff$.
As a result, $t=1$ and $z=1\otimes b_{1}\in B$ As a result, $z\in Z(B)$. 
\end{proof}
\selectlanguage{american}%
\begin{lem}
\label{cor:specific-shape}Suppose $A$ is $G$-primitive $\ff$-algebra
acting on the $G$-graded module $V$ irreducibly and faithfully.
Let $\mathbf{D}:=\endo{}_{A}^{G}(V)$ $D=\mathbf{D}_{e}$, $H=\mbox{supp}\mathbf{D}$
and $V_{T}=\bigoplus_{t\in T}V_{t}$, where $T$ is a set of representatives
of right cosets of $H$ inside $G$. Suppose further that $V_{T}$
is finite dimensional over $D$.

Then, $A$ is a $G$-simple algebra over the field $\kk=Z(A)_{e}$
and $\dim_{\kk}A_{e}<\infty$. Furthermore, $A=M_{\gg}(\mathbf{D})$
for some \foreignlanguage{english}{$\gg=(g_{1},...,g_{n})\in G^{n}$
and a finite dimensional $G$-division $\kk$-algebra $\mathbf{D}$. }
\end{lem}
\begin{proof}
Consider the map $\phi:A^{op}\to\endo{}_{\mathbf{D}}^{G}(V)$ given
by right multiplication. It is easy to check that this is a $G$-graded
homomorphism. Since the action is faithful the map is an embedding. 

By \lemref{AOppositeSimple} $A^{op}$ is $G$-simple and thus also
$G$-primitive. Moreover, suppose $v_{1},...,v_{m}$ is a homogeneous
basis of $V_{T}$. Then any homogeneous $f\in\endo{}_{\mathbf{D}}(V)$
is completely determined by its action on $v_{1},...,v_{m}$. Write
$u_{1}=f(v_{1}),...,u_{l}=f(v_{m})$. Then clearly $(\deg u_{i})^{-1}\cdot\deg v_{i}$
is the same for $i=1...m$. Hence, by \thmref{Density}, there is
an $a\in A^{op}$ such that $\phi(a)=f$. Therefore, $\phi$ is a
$G$-graded isomorphism. 

It is easy to check that 
\[
\endo{}_{\mathbf{D}}(V)\backsimeq\mathbf{D}\otimes_{D}\endo{}_{D}(V_{T})\backsimeq\mathbf{D}\otimes_{D}M_{\gg}(D)\backsimeq\mathbf{D}\otimes_{Z(D)}M_{\gg}(Z(D)),
\]
where $\gg=(t_{1},...,t_{1},...,t_{l},...,t_{l})$, where $\{t_{1},...,t_{l}\}=T$
and each $t_{i}$ appears exactly $\dim_{D}V_{t_{i}}$ times in $\gg$.
Now, notice that if $e=g_{i}^{-1}hg_{j}$ (here $h\in H$ and $g_{i},g_{j}\in T$),
then $g_{i}=hg_{j}$. Therefore, $g_{i}=g_{j}$ and $h=e$ (recall
that $T$ is a transveral of $H$ in $G$). As a result, $A_{e}^{op}\subseteq D\otimes_{D}\endo{}_{D}(V_{T})$.
So $\kk=Z(A^{op})_{e}=Z(D)$, which is a field (since $D$ is a division
ring). All in all, $A$ is $G$-isomorphic to $M_{\gg}(\mathbf{D}^{op})$
which is $G$-simple due to \lemref{AOppositeSimple} and \noteref{MgBSimple}.
\end{proof}
For the next Corollary we will need the following definition.
\begin{defn}
A $G$-graded $\ff$-algebra $A$ is called \emph{$G$-Artinian} if
every descending sequence of $G$-graded ideals of $A$: $I_{1}\supseteq I_{2}\supseteq\cdots$
eventually stabilizes.
\end{defn}
\begin{note}
If $\dim_{\ff}A<\infty$, it is clear that $A$ is $G$-Artinian. 
\end{note}
\begin{thm}[$G$-graded Wedderburn's theorem]
\label{thm:graded-wedderburn} Suppose $A$ is $G$-simple and (left)
$G$-Artinian $\ff$-algebra. Then, $A=M_{\gg}(\mathbf{D})$, where
\foreignlanguage{english}{$\gg=(g_{1},...,g_{n})\in G^{n}$ and $\mathbf{D}$
is a finite dimensional $G$-division $\ff$-algebra. If moreover
}$\kk=Z(A)_{e}$ is algebraically closed,\foreignlanguage{english}{
then $\mathbf{D}=\ff^{\alpha}H$, where $H=\mbox{supp}\mathbf{D}$
and $\alpha\in Z^{2}(H,\ff^{\times})$. }
\end{thm}
\begin{proof}
We want to use \corref{specific-shape} with $V=A$ (as a $G$-graded
left $A$-module). Suppose $V_{T}$ is of infinite dimension over
$\ff=D$ (see \lemref{StructureOfGDivision}). The set 
\[
R_{n}=\{a\in A|\,aU_{n}=0\},
\]
where $U_{n}$ is an $n$-dimensional $G$-vector space inside $V_{T}$,
is clearly a left $G$-graded ideal of $A$. By choosing $U_{1}\subseteq U_{2}\subseteq\cdots$
we get $R_{1}\supseteq R_{2}\supseteq\cdots$. Moreover, by irreducibility
we know that these are strict inclusions. Hence, we reached a contradiction. 

The statement now follows from \corref{specific-shape} and \lemref{StructureOfGDivision}.
\end{proof}
\begin{cor}[$G$-graded Kapalansky]
\label{cor:kapalnsky}If $A$ is $G$-primitive $\ff$-algebra and
moreover satisfies an ordinary PI of degree $d$, then the conclusion
of the previous corollary holds. Furthermore, $\kk=Z(A)_{e}$ is a
field and $\dim_{\kk}A_{T}\leq\frac{d}{2}$.
\end{cor}
\begin{proof}
We use the notation of \thmref{Density}. It is enough to show that
$\dim_{D}V_{T}\leq d/2$. Indeed, otherwise there is an $d/2<n$-dimensional
$G$-graded $D$-subspace $U$ of $V_{T}$ and consider the subring
of $A$ given by 
\[
R=\{a\in A|\,aU\subseteq U\}.
\]
By the previous theorem there is an \textbf{epimorphism} from $R$
to $\endo{}_{D}(U)=M_{n}(D)$. However, $M_{n}(\ff)$ satisfies a
PI of degree $d<2n$ or higher. This contradicts the assumption that
$A$ satisfies a PI of degree $d$.
\end{proof}
\begin{thm}
\label{thm:GsimpleAndPI}Suppose $A$ is a $G$-simple PI $\ff$-algebra,
where $G$ is finite. Then, there is a field $\ff\subseteq\kk$ and
a $G$-graded $\kk$-algebra $B=\kk^{\alpha}H\otimes M_{\gg}(\kk)$
such that 
\[
id_{G,\ff}(A)=id_{G,\ff}(B).
\]
\end{thm}
\begin{proof}
By \corref{kapalnsky}, for $\ll=Z(A)_{e}$, we get that $A$ is finite
dimensional over $\ll$ and $A=M_{\gg}(\mathbf{D})$ for a $G$-division
(finite dimensional) $\ll$-algebra $\mathbf{D}$. As a result, for
$\kk=\overline{\ll}$, $\kk\otimes_{\ll}A\backsimeq M_{\gg}(\kk\otimes_{\ll}\mathbf{D})$
(notice that $\ll$ being $e$-homogeneous implies that the above
isomorphism preserves the $G$-grading). Furthermore, by \lemref{StructureOfGDivision},
$\kk\otimes_{\ll}\mathbf{D}$ is of the form $\kk^{\alpha}H$, where
$H$ is a subgroup of $G$ and $\alpha\in Z^{2}(H,\kk^{\ast})$. Finally
the theorem follows since 
\[
id_{G,\ff}(A)=id_{G,\ff}(\kk\otimes_{\ll}A).
\]
\end{proof}
\begin{rem}
Notice that $\kk\otimes_{\ff}id_{G,\ff}(B)\subseteq id_{G,\kk}(B)$.
However, the other inclusion is not assured. For instance, let $\ff=\mathbb{Q}$,
$\kk=\mathbb{C}$, $G=C_{3}\times C_{3}=\left\langle \sigma\right\rangle \times\left\langle \tau\right\rangle $
and $B=\kk^{\alpha}G$, where $u_{\sigma}u_{\tau}=\xi u_{\tau}u_{\sigma}$
(here $\xi=\sqrt[3]{1}$) and $u_{\sigma}^{3}=u_{\tau}^{3}=1$, we
obtain that $x_{1.\sigma}x_{2,\tau}-\xi x_{2.\tau}x_{1,\sigma}\in id_{G,\kk}(B)$
but not in $\kk\otimes_{\ff}id_{G,\ff}(B)$.
\end{rem}
\begin{defn}
For a $G$-graded algebra $A$ we denote by $J_{G}(A)$ the intersection
of all of its $G$-primitive ideals. We call $J_{G}(A)$ \emph{the
$G$-graded Jacobson radical of $A$}. 

If $J_{G}(A)=\{0\}$, we say that $A$ is $G$-\emph{semisimple}.
\end{defn}
\begin{thm}[see \cite{Cohen1984}]
\label{thm:GSemisimpleAsSemisimple}If $G$ is a finite group and
$\ff$ is of zero characteristic, then $J_{G}(A)=J(A)$ for every
$G$-graded $\ff$-algebra $A$. In particular, $A$ is $G$-semisimple
if and only if $A$ is semisimple.
\end{thm}
\begin{lem}
\label{lem:GSemisimpleIdeals}Suppose $A$ is a $G$-semisimple PI
$\ff$-algebra and $I$ is a $G$-graded ideal of $A$. Then $I$
is also $G$-semsimple \textbf{algebra}. Furthermore, $Z(I)\subseteq Z(A)$. 
\end{lem}
\begin{proof}
By assumption $A$ is a subdirect product of $G$-primitive algebras
$A_{i}=A/M_{i}$, where $M_{i}$ is a $G$-primitive ideal ($i\in\jj$
- an index set). That is, there is a $G$-graded embedding 
\[
\phi:A\to\prod_{i\in\jj}A_{i}
\]
such that $p_{i}\circ\phi$ is surjective for every $i\in\jj$, where
$p_{i}:\prod_{j\in\jj}A_{j}\to A_{i}$ is the natural projection.
By \corref{kapalnsky}, we get that each $A_{i}$ is, in fact, $G$-simple
algebra.

Since every $A_{i}$ is $G$-simple, for every $i\in\jj$ $p_{i}\circ\phi(I)=0$
or $A_{i}$. Denote by $\jj_{0}$ all the indexes $j\in\jj$ for which
the result is $A_{j}$. Hence, the sequance 
\[
0\to I\to\prod_{i\in\jj_{0}}A_{i}\to A_{i}\to0
\]
is exact. This shows that $I$ is a $G$-semisimple algebra.

If $a\in Z(I)$, then $p_{i}\circ\phi(a)$ is, in any case, inside
$Z(A_{i})$ (for every $i\in\jj$). As a result $a\in Z\left\{ \prod_{i\in\jj}A_{i}\right\} \cap A\subseteq Z(A)$.
\end{proof}
\selectlanguage{english}%

\section{\label{sec:Preliminaries}$G$-graded PI }

\selectlanguage{american}%
We will use from now on the standard notations of $G$-graded PI theory. 

\selectlanguage{english}%
A $G$-graded polynomial is an element of the non-commutative free
algebra $F\left\langle X_{G}\right\rangle $ generated by the variables
from $X_{G}=\{x_{i,g}|i\in\mathbb{N},\,g\in G\}$. It is clear that
$\ff\left\langle X_{G}\right\rangle $ is itself $G$-graded, where
$\ff\left\langle X_{G}\right\rangle _{g}$ is the span of all monomials
$x_{i_{1},g_{1}}\cdots x_{i_{n},g_{n}}$, where $g_{1}\cdots g_{n}=g$. 

We say that a $G$-graded polynomial $f(x_{i_{1},g_{1}},...,x_{i_{n},g_{n}})$
is a $G$-graded identity of a $G$-graded algebra $A$, if $f(a_{1},...,a_{n})=0$
for all $a_{1}\in A_{g_{1}},...,a_{n}\in A_{g_{n}}$. The set of all
such polynomials is denoted by $id_{G}(A)$ and it is evident that
it is an ideal of $\ff\left\langle X_{G}\right\rangle $. In fact,
it is a $G$-graded $T$-ideal, that is an ideal which is closed to
$G$-graded endomorphisms of $\ff\left\langle X_{G}\right\rangle $.

For $G$ finite. it is often convenient to view any ungraded polynomial
as $G$-graded via the embedding 
\[
\ff\left\langle X\right\rangle \to\ff\left\langle X_{G}\right\rangle ,
\]
where $x_{i}\mapsto\sum_{g\in G}x_{i,g}$. This identification respects
the $G$-graded identities, in the sense that if $f\in id(A)$ (the
$T$-ideal of ungraded identities of $A$), then $f\in id_{G}(A)$.

As in the ungraded case, the set of ($G$-graded) multilinear identities
of $A$ 
\[
\left\{ \sum_{\sigma\in S_{n}}\alpha_{\sigma}x_{i_{\sigma(1)},g_{\sigma(1)}}\cdots x_{i_{\sigma(n)},g_{\sigma(n)}}\in id_{G}(A)|\,\alpha_{\sigma}\in\ff;\,\,i_{1},...,i_{n}\mbox{ are all distinct}\right\} 
\]
$T$-generates $id_{G}(A)$. 

Suppose $f\in\ff\left\langle X_{G}\right\rangle $, $A$ is a $G$-graded
algebra and $S\subseteq A$ . We write $f(S)$ for the subset of $A$
consisting of all the graded evaluations of $f$ by elements of $S$.
Moreover, if $\mathcal{B}$ is a subset of $A$ consisting of homogenous
elements such that $\mbox{sp}\mathcal{B}=A$, then $\mbox{sp}f(\mathcal{B})=A$.
Thus, $f\in id_{G}(A)$ if and only if $f$ vanishes under graded
substitutions by elements from $\mathcal{B}$. 

\subsection{Central polynomials}

In this section we introduce the main tool for proving Posner's theorem
(see \secref{-graded-posner's-theorem}). 
\selectlanguage{american}%
\begin{defn}
Let $A$ be an $\ff$-algebra. We say that $f\in\fxx$ is a \emph{central
polynomial }of $A$ if $f(A)\subseteq Z(A)$ and $f$ is not a PI
of $A$.

If $A$ is also graded by a group $G$, then we say that $f\in\fxg$
is a\emph{ $G$-graded central polynomial} of $A$ if $f(A)\subseteq Z(A)$
and $f$ is not a $G$-graded PI of $A$. However, we will most of
the time use the phrase ``central polynomials'' also for $G$-graded
central polynomials.
\end{defn}
\selectlanguage{english}%
Suppose $\ff$ is a field of characteristic zero. Then, the following
polynomial (called the Regev polynomial)
\begin{eqnarray*}
F\left\langle X\right\rangle \ni L_{d}(x_{1},...,x_{d},y_{1},...,y_{d}) & = & \sum_{\sigma,\tau\in S_{d}}(-1)^{\sigma\tau}x_{\sigma(1)}y_{\tau(1)}x_{\sigma(2)}x_{\sigma(3)}x_{\sigma(4)}y_{\tau(2)}y_{\tau(3)}y_{\tau(4)}\cdots\\
 &  & \cdots x_{\sigma((n-1)^{2}+1)}\cdots x_{\sigma(d)}y_{\tau((n-1)^{2}+1)}\cdots y_{\tau(d)}
\end{eqnarray*}
is a central polynomial of $M_{n}(\ff)$, where $\exp(A)=d=n^{2}$.
That is, the values of $L_{d}$ when evaluated on $M_{n}(\ff)$ are
in the center of $M_{n}(\ff)$ and there are non-zero values (see
\cite{Formanek1987}).
\begin{rem}
A basic property of $L_{d}$ polynomial is that for $a_{1},...,a_{d},b_{1},...,b_{d}\in M_{n}(F)$
the value of $L_{d}(a_{1},...,a_{d},b_{1},...,b_{d})$ is non-zero
if and only if each one of the sets $\{a_{1},...,a_{d}\}$ and $\{b_{1},...,b_{d}\}$
is $F$-independent. 
\end{rem}
\begin{lem}
\label{lem:CentralAndPosner}Suppose $A$ is a semiprime $\ff$-algebra
(see \defref{semiprime}) of exponent $\exp(A)=d$, where $\ff$ is
a characteristic zero field. Then, $d=n^{2}$ is a square and $A$
is PI equivalent to $M_{n}(\ff)$. Furthermore, if $f$ is a central
polynomial of $A$ it is an identity of any semiprime algebra $A^{\prime}$
of lower exponent $\exp(A^{\prime})=m^{2}$.
\end{lem}
\begin{proof}
The first part is a consequence of the ungraded Posner's Theorem (see
Theorem 1.12.1 in \cite{Giambruno2005}). 

Furthermore, let $f$ be any central polynomial of $M_{n}(\ff)$.
Since for $m<n$ we may view $M_{m}(\ff)$ as a subalgebra of $M_{n}(\ff)$
via the embedding $A\mapsto\begin{pmatrix}A & 0 & \cdots & 0\\
0 & 0\\
\vdots &  & \ddots\\
0 &  &  & 0
\end{pmatrix}$ (notice that the in embedding does not preserve the unit element
of $M_{m}(\ff)$), so $f$ is an identity of $M_{m}(\ff)$.
\end{proof}
We turn to $G$-graded polynomials. 
\begin{defn}
For $f\in\fxg$ and $g\in G$, denote by $\rho_{g}(f)$ the $g$-component
of $f$. When $g=e$, we write $\rho$ for $\rho_{e}$.
\end{defn}
\selectlanguage{american}%
In order to deal with the infinite group case we will need a stronger
notion of ungraded central polynomials.
\begin{defn}
A polynomial $f=f(x_{1},...,x_{l})\in\fxx$ is called \emph{strong
central polynomial }for a $G$-graded $\ff$-algebra $A$ (or \emph{$G$-strong
central polynomial}), if:

\begin{enumerate}
\item $f$ is a central polynomial of $A$.
\item For every $g_{1},...,g_{l}\in G$ and $a_{g_{1}}\in A_{g_{1}},...,a_{g_{l}}\in A_{g_{l}}$
such that $f_{1}=f(x_{1,g_{1}},...,x_{l,g_{l}})$ is non-zero when
evaluated by $\overline{x}_{1,g_{1}}=a_{g_{1}},...,\overline{x}_{l,g_{l}}=a_{g_{l}}$,
then $\rho(f_{1})(a_{g_{1}},...,a_{g_{l}})\neq0$. 
\end{enumerate}
\end{defn}
\selectlanguage{english}%
The following Lemma shows that if $f$ is a strong central polynomial,
then we can construct from it \textbf{$e$-homogenous} $G$-graded
central polynomials.
\begin{lem}
\label{lem:CentralToGraded}Suppose $f$ is $G$-graded central polynomial
of a $G$-graded $\ff$-algebra $A$. Then, $\rho(f)$ is either central
or a $G$-graded identity of $A$.
\end{lem}
\begin{proof}
Decompose $f$ into a sum of its $G$-graded components: $f=\sum_{g\in G}\rho_{g}(f)$.
Fix $h\in G$ and consider $y_{h}\in X_{G}$ which does not appear
in $f$. Since $f$ is central, it is clear that $[y_{h},f]\in id^{G}(A)$.
Hence, $\sum_{g\in G}[y_{h},\rho_{g}(f)]=0$. Since $[y_{h},\rho_{g}(f)]\in\fxg_{h}$
if and only if $g=e$ (notice that $[y_{h},\rho_{g}(f)]\in\fxg_{gh}\oplus\fxg_{hg}$),
we get $[y_{h},\rho_{e}(f)]\in id^{G}(A)$. Because this holds for
all $h\in G$, we conclude that $\rho(f)(A)$ lies inside the center
of $A$.
\end{proof}
\begin{cor}
\label{cor:usufull}Suppose $f=f(x_{1},...,x_{l})$ is a strong central
polynomial for the $G$-graded $\ff$-algebra $A$. Then, if $f_{1}=f(x_{1,g_{1}},...,x_{l,g_{l}})$
is not an identity of $A$ (here $g_{1},...,g_{l}\in G$), $\rho(f_{1})$
is a central polynomial of $A$.
\end{cor}
We will be interested in graded central polynomials for the $G$-graded
algebra $A=\ff^{\alpha}H\otimes M_{\gg}(\ff)$ (as in \defref{matrix-grading}).
We introduce two methods of constructing them. 
\selectlanguage{american}%

\subsubsection{Strong central polynomials using involution.}

Consider the following involution on the free algebra $\ff\left\langle X\right\rangle $
given by: $(\alpha x_{i_{1}}\cdots x_{i_{m}})^{t}=\alpha x_{i_{n}}\cdots x_{i_{1}}$.
We use this action to construct new (ungraded) central polynomials:
Suppose $f=f(x_{1},...,x_{m})\in\fxx$ is a central polynomial of
$M_{n}(\ff)$. Since the transpose action on matrices is an involution
on $M_{n}(\ff)$, we conclude that $f^{t}$ is also a central polynomial
of $M_{n}(\ff)$. Hence, $\check{f}=f(x_{1},...,x_{m})f(y_{1},...,y_{m})^{t}$
is also a central polynomial.
\begin{lem}
Consider the $G$-graded algebra $A=\ff^{\alpha}H\otimes M_{\mathfrak{g}}(\ff)$,
where $\ff=\mathbb{C}$ and the values $\alpha$ takes are roots of
unity. Let $f=L_{d}(x_{1},...,x_{d},y_{1},...,y_{d})\in\fxx\subseteq\fxg$
be the Regev polynomial, where $d=\exp(A)$. Then, for every $a_{1},...,a_{d}\in A$,
the $e$ part of 
\[
\check{f}(a_{1},...,a_{d},a_{1},...,a_{d};a_{1}^{\ast},...,a_{d}^{\ast},a_{1}^{\ast},...,a_{d}^{\ast})
\]
is not zero if and only if $f(a_{1},...,a_{d},a_{1},...,a_{d})\neq0$,
where $\left(cu_{h}\otimes M\right)^{\ast}=\bar{c}u_{h}^{-1}\otimes M^{\ast}$.
(Here $\bar{c}$ denotes the complex conjugate of $c\in\mathbb{C}$
and $M^{\ast}=\overline{M}^{t}$ for $M\in M_{m}(\mathbb{C})$.)
\end{lem}
\begin{proof}
First, we claim that $()^{\ast}$ is an anti-automorphism of $A$.
Indeed, 
\[
\left(cu_{h}\otimes M\cdot c^{\prime}u_{h^{\prime}}\otimes M^{\prime}\right)^{\ast}=\overline{cc^{\prime}}\overline{\alpha(h,h^{\prime})}u_{hh^{\prime}}^{-1}\otimes M^{\prime\ast}M^{\ast}
\]
 whereas 
\[
\left(cu_{h}\otimes M\right)^{\ast}\left(c^{\prime}u_{h^{\prime}}\otimes M^{\prime}\right)^{\ast}=\overline{cc^{\prime}}u_{h}^{-1}u_{h^{\prime}}^{-1}\otimes M^{\prime\ast}M^{\ast}.
\]
Finally, 
\[
u_{h}^{-1}u_{h^{\prime}}^{-1}=(\alpha(h,h^{\prime})u_{hh^{\prime}})^{-1}=\overline{\alpha(h,h^{\prime})}u_{hh^{\prime}}^{-1}.
\]
The last equality follows from $|\alpha(h,h^{\prime})|=1$.

Now we are ready to prove the Lemma. If $f(a_{1},...,a_{d},a_{1},...,a_{d})=0$,
it is obvious that also 
\[
\check{f}(a_{1},...,a_{d},a_{1},...,a_{d};a_{1}^{\ast},...,a_{d}^{\ast},a_{1}^{\ast},...,a_{d}^{\ast})=0,
\]
in particular its $e$ component must be zero.

If, on the other hand, $f(a_{1},...,a_{d},a_{1},...,a_{d})$ is non-zero,
since $\check{f}$ is a central polynomial of $A$, it must be equal
to 
\[
\sum_{g\in H}\alpha_{g}u_{g}\otimes I\in A,
\]
where at least one of the $\alpha_{i}\neq0$. Since $()^{\ast}$ is
an anti-automorphism,
\[
\check{f}(a_{1},...,a_{d},a_{1},...,a_{d};a_{1}^{\ast},...,a_{d}^{\ast},a_{1}^{\ast},...,a_{d}^{\ast})=\left(\sum_{g}\alpha_{g}u_{g}\right)\left(\sum_{g}\bar{\alpha}_{g}u_{g^{-1}}\right)\otimes I.
\]
The $e$ part of the last expression is equal to $\sum_{g\in H}|\alpha_{g}|^{2}>0$. 
\end{proof}
\begin{thm}
\label{thm:PassToC}Suppose that $\ff$ is any field of characteristic
zero, $G$ is a finite group and $A=\ff^{\alpha}H\otimes M_{\mathfrak{g}}(\ff)$
is $G$-graded as in \defref{matrix-grading}. Suppose further that
$f$ is a central polynomial of $A$. Then, $\check{f}$ is strong
central polynomial of $A$.
\end{thm}
\begin{proof}
By extending scalars, we may assume that $\ff$ is an algebraically
closed field containing $\mathbb{C}$. Hence, since every $\alpha\in Z^{2}(H,\mathbb{\ff}^{\ast})$
is cohomologous to $\alpha^{\prime}$ (over the algebraically closed
field $\ff$) which satisfies the assumption in the previous theorem.
Moreover, by Lemma 1.3 in \cite{Aljadeff2014}, the $G$-graded algebras
$\mathbb{\ff}^{\alpha}H\otimes M_{\gg}(\mathbb{\ff})$ and $\ff^{\alpha^{\prime}}H\otimes M_{\gg}(\ff)$
are $G$-graded isomorphic.

As a result $A_{\mathbb{C}}=\mathbb{C}^{\alpha}H\otimes M_{\mathfrak{g}}(\mathbb{C})$
is well defined and $\ff\otimes_{\mathbb{C}}A_{\mathbb{C}}$. By the
previous Lemma, $\check{f}$ is a strong central polynomial of $A_{\mathbb{C}}$,
hence also if $A$.
\end{proof}

\subsubsection{Strong central polynomials using projective representations of groups.}

The following method is due to Eli Aljadeff.

\selectlanguage{english}%
We start with a preliminary Lemma:
\selectlanguage{american}%
\begin{lem}
\label{lem:lifting}Suppose $A$ and $B$ are semisimple and finite
dimensional $\ff$-algebras. If $\phi:A\to B$ is an epimorphism,
then every central minimal idempotent in $B$ can be lifted in $A$.
\end{lem}
\begin{proof}
The kernel of $\phi$, denoted by $I$, must be of the form (as any
ideal of $A$) $A_{i_{1}}\times\cdots\times A_{i_{s}}$. Without loss
of generality assume that $I=A_{q^{\prime}+1}\times\cdots\times A_{r}$.
Thus, the restriction $\phi|_{A_{1}\times\cdots\times A_{q^{\prime}}}$
is an isomorphism. From Corollary 7.50 in ???, it follows that there
is $q=q^{\prime}$ and there is $\sigma\in S_{q}$ such that $\phi(A_{i})=B_{\sigma(i)}$.
The Lemma is now clear.
\end{proof}
\begin{lem}
Assume $\ff$ is a characteristic zero algebraically closed field
and $0\neq\rho=\sum_{g\in G}\mu_{g}u_{g}\in\ff^{\alpha}G$ is a central
idempotent. Then $\mu_{e}\neq0$. 
\end{lem}
\begin{proof}
By a theorem of Schur (see Theorem 11.17 in ???) we know that there
exist a finite central extension 
\[
\xymatrix{0\ar[r] & C\ar[r] & \Gamma\ar[r]^{\pi} & G\ar[r] & 0}
\]
and $\lambda\in\hom(C,\ff^{\ast})$ such that $[\lambda\circ\beta]=[\alpha]\in H^{2}(G,\ff^{\ast})$,
where $[\beta]\in H^{2}(G,C)$ corresponds to the above extension.
Hence, by choosing a transversal $\pi(\gamma_{1})=g_{1},...,\pi(\gamma_{m})=g_{m}$
(here $G=\{e=g_{1},g_{2},...,g_{m}\}$), we get an epimorphism 
\[
\phi:\ff\Gamma\to\ff^{\alpha}G
\]
given by $\phi(u_{c\gamma_{i}})=\lambda(c)u_{g_{i}}$. 

Since both algebras are finite dimensional and semisimple, by \lemref{lifting},
there is a central idempotent $\rho^{\prime}\in\ff\Gamma$ which projects
(via $\pi$) onto $\rho$. Thus, there exists a character $\chi$
of $\Gamma$ such that: 
\[
\rho^{\prime}=\frac{\chi(e)}{|\Gamma|}\sum_{\gamma\in\Gamma}\overline{\chi(\gamma)}u_{\gamma}.
\]
Since $C\leq Z(\Gamma)$ and $\chi$ is a character of an irreducible
representation $V$ of $\Gamma$ we know, by Schur's Lemma, that $c$
acts on $V$ by multiplication by $\frac{1}{\dim V}\chi(c)$. Hence,
$\chi(c\gamma)=\chi(c)\chi(\gamma)$. This all boils down to: 
\[
\rho=\pi(\rho^{\prime})=\left(\sum_{c\in C}\overline{\chi(c)}\lambda(c)\right)\cdot\frac{\chi(e)}{|\Gamma|}\sum_{i=1}^{m}\overline{\chi(\gamma_{i})}u_{g_{i}}.
\]
Since $e\neq0$, $\chi(e)\sum_{c\in C}\overline{\chi(c)}\lambda(c)\neq0$.
Hence also $\mu_{e}=\mu_{g_{1}}=\left(\sum_{c\in C}\overline{\chi(c)}\lambda(c)\right)\cdot\frac{|\chi(e)|}{|\Gamma|}\neq0$.
\end{proof}
\begin{thm}
\label{thm:Central!}Let \foreignlanguage{english}{$A=\ff^{\alpha}H\otimes M_{\gg}(\ff)$},
where $\ff$ is a field of zero characteristic. Suppose $F=F(x_{1},...,x_{l})$
is a central polynomial of $A$. Then, $F$ is a strong central polynomial
of $A$.
\end{thm}
\begin{proof}
Since $\ff^{\alpha}H$ is semisimple (since $\ff$ is of zero characteristic)
and finite dimensional, we can write it as the product $M_{k_{1}}(\ff)\times\cdots\times M_{k_{l}}(\ff)$,
where assuming $k_{1}\geq k_{2}\geq\cdots\geq k_{l}$. So $A$ as
an ungraded algebra is isomorphic to 
\[
\left(M_{k_{1}}(\ff)\times\cdots\times M_{k_{l}}(\ff)\right)\otimes M_{|\gg|}(\ff)=M_{n=k_{1}+|\gg|}(\ff)\times\cdots\times M_{k_{l}+|\gg}(\ff).
\]
As a result, $A$ is PI equivalent to $M_{n}(\ff)$. So, $F$ is a
central polynomial of $M_{n}(\ff)$. Thus, $W=F\left(M_{k_{1}}(\ff)\otimes M_{|\mathfrak{g}|}(\ff)=M_{n}(\ff)\right)=\ff\rho\otimes I$\textbackslash{},
where $\rho=1_{M_{k_{1}}(\ff)}$ is a central idempotent of $\ff^{\alpha}H$.
By the previous Lemma, the $e$-component of $W$ is not zero, hence
$f$ is a non-identity. Thus, by \lemref{CentralToGraded}, $f$ is
a central polynomial of $B$ of degree $e$.
\end{proof}

\subsubsection{$G$-graded central polynomials for lower exponent algebras.}

We start with a definition:
\begin{defn}
A $G$-graded polynomial $f$ is called $d$-\emph{sharp, if for every
$G$-simple algebra $A$ the following holds:}

\begin{itemize}
\item $\exp(A)=d\Longrightarrow f\notin id_{G}(A)$.
\item $\exp(A)<d\Longrightarrow f\in id_{G}(A)$.
\end{itemize}
\end{defn}
\begin{thm}
\label{thm:NeededCentralPolynomials}Suppose $A$ is a $G$-simple
$\ff$-algebra, \foreignlanguage{english}{where $\ff$ is a characteristic
zero field,} which is ungraded PI. Every ungraded central polynomial
$f$ of $A$ is $G$-strong and $\rho(f)$ is $\exp(A)$-sharp.
\end{thm}
\begin{proof}
By \thmref{GsimpleAndPI}, there is a field $\ff\subseteq\kk$ and
a $G$-graded $\kk$-algebra $B=\kk^{\alpha}H\otimes M_{\gg}(\kk)$
such that $id_{G,\ff}(A)=id_{G,\ff}(B)$. Hence, by \thmref{Central!},
$f$ is $G$-strong central polynomial of $A$. 

The second part follows from \lemref{CentralAndPosner} since every
$G$-simple algebra over a characteristic zero field is semisimple
(hence semiprime).
\end{proof}
\selectlanguage{english}%

\section{\label{sec:-graded-posner's-theorem}$G$-graded Posner's theorem}

In order to introduce the main theorem we need the following definitions.
\begin{defn}
\label{def:semiprime}Let $G$ be any group and suppose that $A$
is a $G$-graded $F$-algebra. We say that $A$ is a \emph{$G$-prime
algebra} if for every two graded ideals $I$ and $J$ such that $IJ$
is equal to the zero ideal, one of them is already the zero ideal.
A graded ideal $P$ of $A$ is called \emph{$G$-prime} if the graded
algebra $A/P$ is $G$-prime.

Furthermore, $A$ is called \emph{$G$-graded semiprime} if the intersection
of all $G$-graded prime ideals of $A$ is zero.
\end{defn}
\begin{rem}
\label{rem:SemiprimeGradedOrNot}It is a simple task to verify that
every $G$-primitive algebra is also $G$-prime. Moreover, by \foreignlanguage{american}{\cite{Cohen1984},
every $G$-semiprime is also (ungraded) semiprime.} 
\end{rem}
In the previous section we established that in the case where $A$
is $G$-primitive and satisfies an ordinary PI, $A$ is, in fact,
of the form $M_{\mathfrak{g}}(F^{\alpha}H)$ given that the field
$F$ is algebraically closed. The $G$-graded version of Posner's
theorem asserts a similar ``rigidity'' result for $G$-prime algebras
(which satisfy a PI). Here is the precise statement: 
\begin{conjecture}
\label{conj:G-Posner}Suppose $G$ is any group and $A$ is a $G$-prime
$F$-algebra which also satisfies a PI, then $S=Z(A)_{e}$ is a domain
and $(A\subseteq)S^{-1}A$ satisfies the conclusion of \corref{specific-shape}
for the field $K=S^{-1}S$ (In particular, this algebra is $G$-simple
and finite dimensional over $K$).
\end{conjecture}
In this section we prove this conjecture for $G$ \textbf{residually
finite}. The first step is to establish the claim for $G$ finite.
For this we follow the footsteps of Rowen (\cite{Rowen1973}) and
prove the following key theorem.
\selectlanguage{american}%
\begin{thm}
\label{thm:G-per-posner}Suppose $A$ is a $G$-graded $\ff$-algebra,
where $G$ is a \textbf{finite} group and $\ff$ is of characteristic
zero. Suppose further that $A$ is $G$-semiprime and ungraded PI.
Then, every $G$-graded ideal $I$ intersects non-trivially $Z(A)_{e}$.
\end{thm}
\begin{proof}
By \cite{Cohen1984} $A[x]=A\otimes_{\ff}\ff[x]$ is $G$-semisimple,
where the $G$-grading is given by $A[x]_{g}=A_{g}[x]$. Since $Z(A[x])=(Z(A))[x]$,
it is suffice to prove the theorem for $A[x]$ and $I[x]$, so we
may assume from the beginning that $A$ is $G$-semisimple. By \lemref{GSemisimpleIdeals},
$I$ is $G$-semisimple and $Z(I)\subseteq Z(A)$. So it is enough
to show that every $G$-semisimple PI algebra $A$ satisfies $Z(A)_{e}\neq\{0\}$.
This will follow if we will show that $A$ has a degree $e$ central
polynomial.

By the proof of \lemref{GSemisimpleIdeals}, $A$ is a subdirect product
of $G$-simple $\ff$-algebras $A_{j}$, where $j\in\mathfrak{J}$.
In other words, the map 
\[
\phi:A\to\prod_{j\in\mathfrak{J}}A_{j}
\]
 is a $G$-graded embedding and $\pi_{i}:A\to A_{i}$ is onto, where
$\pi_{i}:A\to\prod_{j\in\mathfrak{J}}A_{j}\to A_{i}$ is the natural
map.

Notice that for every $i\in\jj$, $id(A)\subseteq id(A_{j})$. So
$\exp(A_{j})\leq\exp(A)$. Choose $j_{0}\in\mathfrak{J}_{0}$ such
that $\exp_{\ff}(A_{j_{0}})$ is maximal among the exponents of all
other algebras corresponding to elements in $\mathfrak{J}_{0}$. Write
$B$ for $A_{j_{0}}$ and $d$ for $\exp(B)$. By \thmref{NeededCentralPolynomials},
$B$ has a strong polynomial $F$ such that $f=\rho(F)$ is $d$-sharp.
That is, $f$ is central for $A_{j_{0}}$ and for every other $j\in\jj$
$f$ is either an identity of $A_{j}$ or a central polynomial of
$A_{j}$. In other words, $f(A)\in Z\left(\prod_{j\in\mathfrak{J}}A_{j}\right)_{e}\cap A\subseteq Z(A)_{e}$.
\end{proof}
\begin{rem}
\label{rem:SeeMe}The proof above shows that if $A$ is $G$-semiprime
($G$ is a finite group), then every ungraded central polynomial $f$
of $A$ is $G$-strong central polynomial of $A$.
\end{rem}

\begin{rem}
\label{rem:OnlyFinite}Notice that for $G$ infinite, it is no longer
true that $A[x]$ is $G$-semisimple (take $G=\mathbb{Z}$ and consider
$A=\ff[x]$ where $\deg x^{n}=n\in\mathbb{Z}$). As a result, in order
to generalize the previous theorem to infinite groups one should introduce
a d idea. 
\end{rem}
As a result we obtain \conjref{G-Posner} for finite groups. Here
is the precise statement.
\begin{cor}
\label{cor:conj-for-finite}Let $G$ be a finite group and $\ff$
a field of characteristic zero. Suppose $A$ is $G$-prime over $\ff$
and satisfies an ordinary PI. Then, $S=Z(A)_{e}^{\times}$ does not
contain zero divisors of $A$. As a result $A$ is $G$-embedded in
$A_{1}=S^{-1}A$ which is a $G$-graded $\kk=S^{-1}Z(A)_{e}$-algebra.
Moreover, $A_{1}$ is finite dimensional $G$-simple $\kk$-algebra.
\end{cor}
\begin{proof}
. Regarding the first part, it is enough to prove that no element
$c\in Z(A)_{e}^{\ast}$ annihilates a non-zero homogenous element
$a\in A$. Suppose otherwise: $ca=0$. Hence, $AcA\cdot AaA$ is a
product of graded ideals which is equal to $A^{3}caA=0$. Thus, $AcA=0$
or $AaA=0$. It is therefore enough to show that $AaA=0$ is impossible
(since it is more general than $AcA=0$). Indeed, in that case $Aa$
is a graded ideal, thus we obtain from $AaA=0$ that $Aa=0$ (since
$A\neq0$). But now $\ff a$ is an ideal and so $\ff a=0$, which
yields $a=0$ - a contradiction.

Suppose $I$ is a non-zero $G$-graded ideal of $A_{1}$. Since $I\cap\kk=I\cap Z(A_{1})\neq0$,
we conclude that $I$ must be equal to $A_{1}$. By \corref{kapalnsky}
we are done.
\end{proof}

\subsubsection{Grading by a quotient group}
\begin{defn}
Let \foreignlanguage{english}{$A$ be an $\ff$-algebra graded by
$G$ and let $Q=G/N$ be any quotient group of $G$. We define the
\emph{induced $Q$-grading} on $A$ by setting 
\[
A_{gN}=\oplus_{h\in N}A_{gn},
\]
for every $g\in G$.}
\end{defn}
\selectlanguage{english}%
\begin{example}
Two notable special cases are when $N=\{e\}$ and $N=G$. In the first
case, we obtain the given $G$-grading on $A$ and in the second we
are in the situation of having no grading at all 
\end{example}
\begin{defn}
Let $G$ be a group and $Q=G/N$ be any quotient group of $G$. Denote
by $\psi_{G,Q}$ the $\ff$-algebra map $\psi_{G.Q}:\fxg\to\ff\left\langle X_{Q}\right\rangle $
induced by $\psi_{G,Q}(x_{i,g})=x_{i,gN}$.

Furthermore. $f\in\ff\left\langle X_{G}\right\rangle $ is said to
be \emph{$Q$-stable }if the diagram:
\[
\xymatrix{\fxg\ar@{->}[r]^{\psi_{G,Q}}\ar@{->}[d]^{\rho_{g}} & \ff\left\langle X_{Q}\right\rangle \ar@{->}[d]^{\rho_{gN}}\\
\ff\left\langle X_{Q}\right\rangle _{e}\ar@{->}[r]^{\psi_{G,Q}} & \ff\left\langle X_{Q}\right\rangle _{eN}
}
\]
commutes for every $g\in G$ \textbf{on $f$}. 
\end{defn}
\begin{example}
Let $G=C_{4}=\left\langle \tau\right\rangle $ and $N=\left\langle \tau^{2}\right\rangle $.
The polynomial $f_{1}=x_{2,\tau^{2}}$ is not $Q$-stable but $f_{2}=2x_{1,e}$
is. 
\end{example}

\subsubsection{Residually finite case.}
\selectlanguage{american}%
\begin{thm}
Let $G$ be a residually finite group and $\ff$ a characteristic
zero field. Suppose $A$ is $G$-semiprime and ungraded PI. Then,
every $G$-graded ideal $I$ intersects non-trivially $Z(A)_{e}$.
\end{thm}
\begin{proof}
By \cite{Cohen1984} $A$ is a semiprime $\ff$-algebra. Thus, $A[x]$
is semisimple. As before we replace $A$ and $I$ by $A[x]$ and $I[x]$
what allows us to assume that $I$ (and $A$) is a semisimple algebra
and $Z(I)\subseteq Z(A)$. So, it will be suffice to show that $I$
has a strong central polynomial. In other words, we need to show that
if $A$ is $G$-graded and semisimple, it possess a strong central
polynomial. 

By \lemref{CentralAndPosner}, $A$ is PI equivalent to $M_{n}(\ff)$,
where $n^{2}=\exp(A)$. Hence Regev's polynomial $F=L_{n^{2}}(x_{1},...,x_{2n^{2}})$
is a central polynomial for $A$ (in fact, we could as well chosen
any other central polynomial of $M_{n}(\ff)$). As a result, there
are $g_{1},...,g_{2n^{2}}\in G$ for which there are $a_{g_{1}}\in A_{g_{1}},...,a_{g_{2n^{2}}}\in A_{g_{2n^{2}}}$
such that $F(a_{g_{1}},...,a_{g_{2n^{2}}})\neq0$. Denote $F_{G}=F(x_{1,g_{1}},...,x_{2n^{2},g_{2n^{2}}})$
and $f=\rho(F_{G})$. We will show that $f$ is central polynomial
of $A$.

Let $N$ be a normal finite index subgroup of $G$ such that all the
distinct products of $\{g_{1},...,g_{2n^{2}}\}$ of length $2n^{2}$
remain distinct in $Q=G/N$. Therefore, $F_{G}$ is $Q$-stable. Regard
$A$ as a $Q$-graded algebra. Hence, by the proof of \corref{conj-for-finite}
(see also \remref{SeeMe}), $\rho_{eN}\left(\psi_{G,Q}(F_{G})\right)$
is $Q$-central polynomial of degree $eN$ for $A$. 

Since $\psi_{G,Q}\left(f\right)=\rho_{eN}\left(\psi_{G,Q}(F_{G})\right)$,
it is clear that 
\[
f(A)\subseteq\left(\rho_{eN}\left(\psi_{G,Q}(F_{G})\right)\left(A\right)\right)\cap A_{e}\subseteq Z(A)_{e}.
\]
Finally, $\psi_{G,Q}(F_{G})(a_{g_{1}},...,a_{g_{2n^{2}}})=F_{G}(a_{g_{1}},...,a_{g_{2n^{2}}})\neq0$.
Hence, the fact that $F$ is $Q$-strong for $A$, forces 
\[
0\neq\rho_{eN}\left(\psi_{G,Q}(F_{G})\right)(a_{g_{1}},...,a_{g_{2n^{2}}})=f(a_{g_{1}},...,a_{g_{2n^{2}}}).
\]
All in all, we shown that $f$ is indeed a $G$-graded central polynomial
of $A$.
\end{proof}
\begin{cor}
\thmref{G-per-posner} holds also when $G$ is a residually finite
group and $\ff$ is of characteristic zero.
\end{cor}

\section{A consequence of the $G$-graded Posner's Theorem. }

In \cite{Aljadeff2014} Aljadeff and Haile proved the following theorem:
\begin{thm}
Suppose $G$ is a finite group and $A_{1},A_{2}$ are $G$-simple
algebras of finite dimension over an algebraically closed field $\ff$
of characteristic zero. Then, $id_{G}(A_{1})=id_{G}(A_{2})$ if and
only if $A_{1}$ and $A_{2}$ are $G$-isomorphic over $\ff$.
\end{thm}
\begin{note}
It is clear that if $A_{1}$ is $G$-isomorphic to $A_{2}$, then
$A_{1}$ and $A_{2}$ share the same $G$-graded identities. The other
direction requires all the work.
\end{note}
The proof they presented is based heavily on an elaborate analysis
of the structure of a $G$-simple algebra given in \thmref{graded-wedderburn}.
Here we present an alternative proof based on the $G$-graded Posner
theorem proven in the previous section. 

To start, recall the construction of the ($G$-graded) generic algebra
over $\ff$. Given a finite dimensional $G$-graded $\ff$-algebra,
chose a $G$-graded $\ff$-basis of $A$: $B=\cup_{g\in G}B_{g}$
($B_{g}$ consists of elements of a basis of $A_{g}$) and consider
the commutative variables 
\[
\Lambda=\left\{ t_{i,g,b}|\,b\in B_{g},\,g\in G,\,i\in\mathbb{N}\right\} .
\]
Finally denote by $U_{A}$ the $\ff$-subalgebra of $A\otimes_{\ff}\ff(\Lambda)$
generated by the elements $y_{i,g}=\sum_{b\in B_{g}}t_{i,g,b}b$.
It is easy to verify that the $G$-homomorphism 
\[
\phi:\fxg/id_{G}(A)\to U
\]
 given by $\phi(x_{i,g})=y_{i,g}$ is well defined and is $G$-isomorphism.
As a result, we may identify $\fxg/id_{G}(A)$ with $U$.

In our case we have $U_{i}\simeq\fxg/id_{G}(A_{i})$ ($i=1,2$) and
since $id_{G}(A_{1})=id_{G}(A_{2})$, we must have $U_{1}=U_{2}$.
As a result, we will denote this algebra (also) by $U$. Notice that
$U\subseteq A_{i}\otimes_{\ff}\ff(\Lambda_{i})$.
\begin{lem}
$U$ is $G$-semisimple.
\end{lem}
\begin{proof}
If $J_{G}(U)\neq0$, then there is a $G$-graded substitution of the
variables $x_{i,g}\in X_{G}$ by elements from $A_{1}$ (of course,
we could as well use $A_{2}$) for which $J_{G}(U)$ is mapped to
a non-zero ideal $I$ in $A_{1}$ and the corresponding map is surjective.
However, it is well known that Jacobson radical is mapped to Jacobson
radical when the map under consideration is surjective. We got a contradiction
to $A_{1}$ being $G$-semisimple.
\end{proof}
\begin{lem}
$Z(U)_{e}\subseteq\ff(\Lambda_{i})\cdot1_{A_{i}}\otimes1$ for $i=1,2$.
As a result, $Z(U_{e})$ does not contain any zero-divisors of $U$.
\end{lem}
\begin{proof}
First observe that $Z(A_{i})_{e}=\ff1_{A_{i}}$. This can be verified
easily using \thmref{graded-wedderburn}. As a result, if $f(y_{1,g_{1}},...,y_{n,g_{n}})\in Z(U_{e})$,
then $\phi(f)\in Z\left(A_{i}\otimes\ff(\Lambda_{i})\right)_{e}=Z(A_{i})_{e}\otimes\ff(\Lambda_{i})=\ff(\Lambda_{i})\cdot1_{A_{i}}\otimes1$.
\end{proof}
\begin{rem}
It is possible to prove that $Z(A_{i})=\ff1_{A_{i}}$ by only using
the classical Wedderburn theorem for (ungraded) semisimple algebras.
Here is a sketch: Set $A=A_{1}$. Since the Jacobson radical is nilpotent
for commutative algebras, one shows that the semisimplicity of $A$
forces $Z(A_{e})$ to be also semisimple. Next, if $e_{1},e_{2}$
are two nonzero orthogonal idempotents in $Z(A)_{e}$, then $Ae_{i}$
is $G$-graded ideal of $A$, so equals $A$. Hence, there is $a\in A$
for which $ae_{1}=e_{2}$. But this is an absurd since $0=e_{1}\cdot e_{2}=a\cdot e_{1}^{2}=ae_{1}=e_{1}$.
Thus $Z(U_{e})$ must be a field. Since surely, $\ff1\subseteq Z(U_{e})$
and $\ff$ is algebraically closed, we must have equality. 
\end{rem}
\begin{notation}
$S=Z(U)_{e}-\{0\}$.
\end{notation}
Due to the previous Lemma, 
\[
U\subseteq S^{-1}U\subseteq A_{i}\otimes\ff(\Lambda_{i}).
\]
Now by \thmref{G-per-posner}, we know that $S^{-1}U$ is $G$-simple
and f.d. over $\kk:=Z(S^{-1}U)_{e}\subseteq\ff(\Lambda_{i})$ (which
is a field). Furthermore, since the graded identities of a $G$-simple
algebra are defined over an algebraically closed field (this follows
from \thmref{GsimpleAndPI} and the proof of \thmref{PassToC}) and
the fact that $id_{G,\ff}(A_{i})=id_{G,\ff}(S^{-1}U)$ forces 
\[
id_{G,\ff(\Lambda_{i})}(S^{-1}U\otimes_{\kk}\ff(\Lambda_{i}))=id_{G,\ff(\Lambda_{i})}(A_{i}\otimes\ff(\Lambda_{i})).
\]

\begin{lem}
The $\ff(\Lambda_{i})$-algebra $S^{-1}U\otimes_{\kk}\ff(\Lambda_{i})$
is $G$-simple. Furthermore, the map 
\[
\nu:S^{-1}U\otimes_{\kk}\ff(\Lambda_{i})\to S^{-1}U\cdot\ff(\Lambda_{i})
\]
is a $G$-graded isomorphism.
\end{lem}
\begin{proof}
The first statement is clear by \propref{tensorGSimple} Moreover,
it is also clear that $\nu$ is a $G$-graded epimomorphism. Finally,
the kernel of $\nu$ (which is a $G$-graded ideal of $S^{-1}U\otimes_{\kk}\ff(\Lambda_{i})$)
is zero, thus proving that $\nu$ is a $G$-graded isomorphism.
\end{proof}
\begin{lem}
$S^{-1}U\cdot\ff(\Lambda_{i})=A_{i}\otimes\ff(\Lambda_{i})$.
\end{lem}
\begin{proof}
We know by now that $S^{-1}U\cdot\ff(\Lambda_{i})$ is a $G$-simple
$\ff(\Lambda_{i})$-subalgebra of $A_{i}\otimes\ff(\Lambda)$ having
the same $G$-graded identities. So 
\[
\dim_{\ff(\Lambda_{i})}S^{-1}U\cdot\ff(\Lambda_{i})=\exp_{G}S^{-1}U\cdot\ff(\Lambda_{i})=\exp_{G}A_{i}\otimes\ff(\Lambda)=\dim_{\ff(\Lambda_{i})}A_{i}\otimes\ff(\Lambda).
\]
Thus the statement follows.
\end{proof}
We have shown that for a field $\ll$ containing both $\ff(\Lambda_{1})$
and $\ff(\Lambda_{2})$, one has
\[
A_{1}\otimes_{\ff}\ll=A_{2}\otimes_{\ff}\ll.
\]

To conclude that $A_{1}$ and $A_{2}$ are $G$-isomorphic we only
need to use the well known fact that if $X$ is any algebraic variety
over a field $(\ff\subseteq)\ll$, then since $\ff$ is algebraically
closed, $X(\ff)$ (the $\ff$points of $X$) is dense in $X$. Indeed,
we take $X$ to be $\text{Hom}_{G,\ll}(A_{1}\otimes_{\ff}\ll,A_{2}\otimes_{\ff}\ll)$.
Notice that $X$ is defined over $\ff$. The subset $V=\text{Iso}_{G,\ll}(A_{1}\otimes_{\ff}\ll,A_{2}\otimes_{\ff}\ll)$
is open (defined by non-vanshing of a certain determinant) and non-empty.
Hence, $V$ contains a point defined over $F$ - finishing the proof.

\bibliographystyle{plain}
\bibliography{C:/Users/Yakov/Dropbox/Thesis/reffe}

\begin{thebibliography}{10}

\bibitem{AlexeiBelov-Kanela}
Louis~Rowen Alexei Belov-Kanel.
\newblock The braun-kemer-razmyslov theorem for affine-pi algebras.
\newblock see http://arxiv.org/pdf/1405.0730.pdf.

\bibitem{Aljadeff2014}
Eli Aljadeff and Darrell Haile.
\newblock Simple {$G$}-graded algebras and their polynomial identities.
\newblock {\em Trans. Amer. Math. Soc.}, 366(4):1749--1771, 2014.

\bibitem{MR2488221}
Yu.~A. Bakhturin, M.~V. Za{\u\i}tsev, and S.~K. Segal.
\newblock Finite-dimensional simple graded algebras.
\newblock {\em Mat. Sb.}, 199(7):21--40, 2008.

\bibitem{MR2072615}
I.~N. Balaba.
\newblock Graded prime {PI}-algebras.
\newblock {\em Fundam. Prikl. Mat.}, 9(1):19--26, 2003.

\bibitem{Cohen1984}
M.~Cohen and S.~Montgomery.
\newblock Group-graded rings, smash products, and group actions.
\newblock {\em Trans. Amer. Math. Soc.}, 282(1):237--258, 1984.

\bibitem{Formanek1987}
Edward Formanek.
\newblock A conjecture of {R}egev about the {C}apelli polynomial.
\newblock {\em J. Algebra}, 109(1):93--114, 1987.

\bibitem{Giambruno2005}
Antonio Giambruno and Mikhail Zaicev.
\newblock {\em Polynomial identities and asymptotic methods}, volume 122 of
  {\em Mathematical Surveys and Monographs}.
\newblock American Mathematical Society, Providence, RI, 2005.

\bibitem{Herstein1994}
I.~N. Herstein.
\newblock {\em Noncommutative rings}, volume~15 of {\em Carus Mathematical
  Monographs}.
\newblock Mathematical Association of America, Washington, DC, 1994.
\newblock Reprint of the 1968 original, With an afterword by Lance W. Small.

\bibitem{Regev1971}
Amitai Regev.
\newblock Existence of polynomial identities in {$A\otimes _{F} B$}.
\newblock {\em Bull. Amer. Math. Soc.}, 77:1067--1069, 1971.

\bibitem{Rowen1973}
Louis Rowen.
\newblock Some results on the center of a ring with polynomial identity.
\newblock {\em Bull. Amer. Math. Soc.}, 79:219--223, 1973.

\end{thebibliography}

\end{document}